\documentclass[3p]{amsart}

\newtheorem{theorem}{Theorem}{}
\newtheorem{lemma}[theorem]{Lemma}

\newtheorem{proposition}[theorem]{Proposition}

\theoremstyle{definition}
\newtheorem{definition}[theorem]{Definition}

\theoremstyle{remark}
\newtheorem{remark}[theorem]{Remark}

\numberwithin{equation}{section}

\begin{document}

\title[Blow-up rates for higher-order parabolic
equations]{Blow-up rates for higher-order a semilinear parabolic
equation with nonlinear memory term}

\author{Ahmad Z. FINO}
\address{Department of Mathematics, Faculty of Sciences, Lebanese University, P.O. Box 1352, Tripoli, Lebanon}
\email{ahmad.fino01@gmail.com; afino@ul.edu.lb}







\begin{abstract}
In this paper, we establish blow-up rates for a higher-order semilinear parabolic
equation with nonlocal in time nonlinearity with no positive assumption on the solution. We also give Liouville-type theorem for higher-order semilinear parabolic equation with infinite memory nonlinear term which plays the main tools to prove our blow-up rate result. Finally, we study the well-posedness of mild solutions.
\end{abstract}

\maketitle

\section{Introduction}
\setcounter{equation}{0}
In this paper, we investigate the higher-order semilinear parabolic
equation with nonlocal in time nonlinearity
\begin{equation}\label{1}
\left\{
\begin{array}{ll}
\,\,\displaystyle{u_t+(-\Delta)^{m}
u=\int_0^t(t-s)^{-\gamma}|u|^{p}\,ds}&\displaystyle {x\in {\mathbb{R}^n},\;t>0,}\\
\\
\displaystyle{u(x,0)=u_0(x)}&\displaystyle{x\in {\mathbb{R}^n},}
\end{array}
\right. \end{equation} where $u_0\in C_0(\mathbb{R}^n),$ $n\geq1,$
$m\in\mathbb{N}^*$, $0<\gamma<1$, $p>1$. The space $C_0(\mathbb{R}^n)$ denotes the space of all continuous functions tending to zero at infinity.

Higher-order semilinear homogeneous equations arise, see e.g. the monograph \cite{Peletier}, in numerous problems in applications such as the higher-order diffusion, the phase transition, the flame propagation, and the thin film theory.

When $m=1$ and $\gamma\to1$, using the relation
$$\lim_{\gamma\to1}c_\gamma\,s_+^{-\gamma}=\delta_0(s)\quad\text{in distributional sense with}\,\,s_+^{-\gamma}:=\left\{\begin{array}{ll}
s^{-\gamma}&\,\,\text{if}\,\,s>0,\\
0&\,\,\text{if}\,\,s<0,
\end{array}
\right.
$$
with $c_\gamma=1/\Gamma(1-\gamma)$, and a suitable change of variable, problem \eqref{1} reduced to the following semilinear heat equation
\begin{equation}\label{heat}
\left\{
\begin{array}{ll}
\,\,\displaystyle{u_t-\Delta u=|u|^{p}}&\displaystyle {x\in {\mathbb{R}^n},\;t>0,}\\
\\
\displaystyle{u(x,0)=u_0(x)}&\displaystyle{x\in {\mathbb{R}^n}.}
\end{array}
\right. \end{equation}
The exponent $p_F=1+2/N$ is known as the 
critical Fujita exponent of \eqref{heat}. Namely, for $p<p_F$, Fujita \cite{Fuj} proved the nonexistence of nonnegative global-in-time solution for any nontrivial initial condition, and for $p>p_F$, global solutions do exist for any sufficiently small nonnegative initial data. The proof of a blow-up of all nonnegative solutions in the critical case $p=p_F$ was completed in \cite{11}. To understand the behavior of the solution near the finite blow-up time, the first step consists in deriving a bound for the blow-up rate. Giga and Kohn \cite{Giga} proved that if $(N-2)p<N+2$ and $u$ is a positive solution of \eqref{heat}, then there  exists a constant $C>0$ such that $u(x,t)\leq C\,(T^*-t)^{-1/(p-1)}$, for all $x\in\mathbb{R}^N$, where $T^*$ is the maximal time of existence.

When  $m=1$ and $0<\gamma<1$, Cazenave et al. \cite{CDW} studied problem \eqref{1} and proved that the critical exponent is 
$$p_*= \Big\{\frac{1}{\gamma},1+\frac{2(2-\gamma)}{(N-2+2\gamma)_+}\Big\},$$
where $(\cdotp)_+$ is the positive part. The study in \cite{CDW} reveals the surprising fact that for equation \eqref{1} the critical exponent in Fujita's sense $p_*$ is not the one predicted by scaling. Moreover, Fino and Kirane \cite{FinoK} derived the blow-up rate estimates for the parabolic equation (\ref{1}). Namely, they
proved that, if $u_0\in
C_0(\mathbb{R}^N)\cap L^2(\mathbb{R}^N),$ $u_0\geq0,$ $u_0\not\equiv0$ and if $u$ is the blowing-up
solution of (\ref{1}) at the finite time $T^*>0,$ then there are constants $c,C>0$ such that
$c(T^*-t)^{-\alpha_1}\leq \sup_{\mathbb{R}^N} u(\cdotp,t)\leq
C(T^*-t)^{-\alpha_1}$ for $1<p\leq1+
2(2-\gamma)/(N-2+2\gamma)_+$ or
$1<p<1/\gamma$ and all $t\in(0,T^*),$ where
$\alpha_1:=(2-\gamma)/(p-1)$.  They used a scaling argument to reduce the problems of blow-up rate to Fujita-type theorems (it is similar to blow-up analysis in elliptic problems to reduce the
problems of a priori bounds to Liouville-type theorems). As far as we know, this method was first applied to parabolic problems by Hu \cite{Hu}, and then was used in various parabolic equations and systems (see \cite{CFila,FQ}). We refer the reader to the excellent paper of Andreucci and Tedeev \cite{tedeev}  for the blow-up rate by an alternative method. 

When $m>1$ and $\gamma\to1$, Galaktionov and Pohozaev \cite{Galaktionov} have shown that $p=1+2m/N$ is the critical exponent of \eqref{1}. Moreover, Pan and Xing \cite{PX} studied this equation and its corresponding system, they derived the blow-up rates of the solution and proved that $\sup_{\mathbb{R}^N} |u(\cdotp,t)|\leq C(T^*-t)^{-1/(p-1)}$ for $1<p\leq1+2m/N$, $m\geq1$, under some condition on the initial data, where $T^*$ is the maximal time of existence.

When  $m>1$ and $0<\gamma<1$, problem \eqref{1} has been considered by Sun and Shi  \cite{SS}. They studied the global existence/nonexistence of solution. Namely, they proved that the critical exponent for \eqref{1} is 
$$p_*= \left\{\frac{1}{\gamma},1+\frac{2m(2-\gamma)}{(N-2m+2m\gamma)_+}\right\}.$$ Our main goal is to derive the blow-up rate estimates for the parabolic equation \eqref{1}. Our proof is similar to the ones in \cite{FinoK} and \cite{PX}. We also prove Liouville-type theorem for \eqref{1} with different memory nonlinear term (see \eqref{2+}), which plays a crucial role to obtain our blow-up rate result. The novelty is that no positive assumption on the solution is needed. Finally, in order to obtain a lower-bound for the blow-up rate, we complete the study of \cite{SS} by proving the local existence of mild solutions for \eqref{1}.

Let us first present our well-posedness result.
\begin{theorem}[Local existence]\label{T0+}
Given $u_0\in C_0(\mathbb{R}^n)$, $0<\gamma< 1$, $m\geq 1$, and $p>1$. There exist a maximal
time $T_{\max}>0$ and a unique mild solution $u\in
C([0,T_{\max}),C_0(\mathbb{R}^n))$ to the problem \eqref{1}. Moreover, either $T_{\max}=\infty$ or else $T_{\max}<\infty$ and in this case $\|u(t)\|_{L^\infty(\mathbb{R}^n)}\rightarrow\infty$ as $t\rightarrow T_{\max}.$ In addition, if $u_0\in C_0(\mathbb{R}^n)\cap L^r(\mathbb{R}^n)$, for $1\leq r<\infty,$ then $u\in
C([0,T_{\max}),C_0(\mathbb{R}^n)\cap L^r(\mathbb{R}^n)).$
\end{theorem}
Next, our main result is the following theorem which present the blow-up rate for the blowing-up solutions to the parabolic problem (\ref{1}).
\begin{theorem}\label{Theorem1}
Let $u_0$ satisfies \eqref{13} below, and
$$
p\leq 1+
2m(2-\gamma)/(N-2m+2m\gamma)_+\qquad\hbox{or}\qquad p<1/\gamma.
$$
If $u$ is the blowing-up mild solution
of \eqref{1} in a finite time $T_{\max}:=T^*$, then
there exist two constants $c,C>0$ such that
\begin{equation}\label{BR+}
   c(T^*-t)^{-\alpha_1}\leq \sup_{\mathbb{R}^n} |u(\cdotp,t)|\leq C(T^*-t)^{-\alpha_1},\qquad t\in(0,T^*),
\end{equation}
where $\alpha_1:=(2-\gamma )/ (p-1)$.
\end{theorem}
Throughout this paper, positive constants will be denoted by $C$ and will change from line to line. 

The remainder of this paper is organized as follows: Section \ref{section2} concerns preliminaries. In Section \ref{section3}, we prove the local existence of the mild solution (Theorem \ref{T0+}) of  \eqref{1}. We devote to the proof of the main result (Theorem \ref{Theorem1}) in Section \ref{section4}. 

\section{Preliminaries}\label{section2}
\setcounter{equation}{0}
In this section, we present some definitions and results that will be used hereafter.\\
Let we start by giving the solution of the following homogenous equation
\begin{equation}\label{HS}
\left\{
\begin{array}{ll}
\,\,\displaystyle{u_t+(-\Delta)^{m}
	u=0}&\displaystyle {x\in {\mathbb{R}^n},\;t>0,}\\
\\
\displaystyle{u(x,0)=u_0(x)}&\displaystyle{x\in {\mathbb{R}^n}.}
\end{array}
\right. \end{equation}
Let $u_0\in X=:C_0(\mathbb{R}^n)$, and $A=-(-\Delta)^{m}$. Using \cite[Theorem~3.7, ~p.~217]{Pazy}, $A$ is the infinitesimal generator of an analytic semigroup $S(t):X\to X$,  $t\geq0$. Therefore, by \cite[Theorem~1.3, ~p.~102]{Pazy} and \cite[Corollary~1.5, ~p.~104]{Pazy}, the initial value problem \eqref{HS} has a unique solution $u(t)=S(t)u_0$, $t\geq0$, which is continuously differentiable on $[0,\infty[$. Moreover, the operator $S(t)$ can be presented as follows (see \cite{Galaktionov})
$$\begin{array}{lrcl}
S(t):& X & \longrightarrow &X \\
& \varphi & \longmapsto & S(t)\varphi=b(\cdotp,t)\ast \varphi,
\end{array}$$
where $b(\cdotp,t)$ denotes the kernel of the operator $S(t)$ (the fundamental solution of the parabolic operator $\partial_t+(-\Delta)^{m}$),  which is presented by
$$
 b(t,x)=\mathcal{F}^{-1}(e^{-|w|^{2m}t})= (2\pi)^{-N}\int_{\mathbb{R}^n}\exp^{-|w|^{2m}t-iw\cdotp x}\,dw.
$$
Furthermore, by \cite[Theorem~3.3]{Cui}, $S(t)$ satisfies the following $L^{p}$-$L^{q}$ estimate
\begin{eqnarray}\label{LpLq}
\|S(t)\varphi \|_q&\leq& Ct^{-\frac{n}{2m}(\frac{1}{p}-\frac{1}{q})}\|\varphi\|_p, 
\end{eqnarray}
for all $t>0$, $\varphi\in L^p(\mathbb{R}^n)$, $1\leq p\leq q\leq\infty$, for some positive constant $C=C(m,n,p,q)$. The kernel $b(\cdotp,t)$ changes sign, when $m>1$, and is oscillatory as $|x|\to\infty$, and the associated semigroup $S(t)$ is not order-preserving. So, there is no comparison principle for \eqref{HS}.\\

Next, we present the tools concerning the fractional integrals and
fractional derivatives.
\begin{definition}(Absolutely continuous functions)\\
A function $f:[a,b]\rightarrow\mathbb{R}$, $-\infty<a<b<\infty$, is absolutely continuous if and only if there exists a Lebesgue summable function $\varphi\in L^1(a,b)$ such that 
$$f(t)=f(a)+\int_{a}^t\varphi(s)\,ds.$$ 
The space of these functions is denoted by $AC[a,b]$. Moreover, we define
$$AC^2[a,b]:=\left\{f:[a,b]\rightarrow\mathbb{R}\;\text{such that}\;f'\in
AC[a,b]\right\}.$$
\end{definition}

\begin{definition}(Riemann-Liouville fractional integrals)\cite[Chapter~1]{SKM}\\
Let $f\in L^1(a,b)$, $-\infty<a<b<\infty$. The Riemann-Liouville left- and right-sided fractional integrals of order $\alpha\in(0,1)$ are, respectively, defined by
\begin{equation}\label{I1}
I^\alpha_{a|t}f(t):=\frac{1}{\Gamma(\alpha)}\int_{a}^t(t-s)^{-(1-\alpha)}f(s)\,ds, \quad t>a,
\end{equation}
and
\begin{equation}\label{I2}
I^\alpha_{t|b}f(t):=\frac{1}{\Gamma(\alpha)}\int_t^{b}(s-t)^{-(1-\alpha)}f(s)\,ds, \quad t<b,
\end{equation}
where $\Gamma$ is the Euler gamma function.
\end{definition}

\begin{definition}(Riemann-Liouville fractional derivatives)\cite[Chapter~1]{SKM}\\
Let $f\in AC[a,b]$, $-\infty<a<b<\infty$. The Riemann-Liouville left- and right-sided fractional derivatives of order $\alpha\in(0,1)$ are, respectively, defined by
\begin{equation}\label{}
D^\alpha_{a|t}f(t):=\frac{d}{dt}I^{1-\alpha}_{a|t}f(t)=\frac{1}{\Gamma(1-\alpha)}\frac{d}{dt}\int_{a}^t(t-s)^{-\alpha}f(s)\,ds, \quad t>a,
\end{equation}
and
\begin{equation}\label{}
D^\alpha_{t|b}f(t):=-\frac{d}{dt}I^{1-\alpha}_{t|b}f(t)=-\frac{1}{\Gamma(1-\alpha)}\frac{d}{dt}\int_t^{b}(s-t)^{-\alpha}f(s)\,ds, \quad t<b.
\end{equation}
\end{definition}

\begin{proposition}(Integration by parts formula)\cite[(2.64)~p.46]{SKM}\\
Let $\alpha\in(0,1)$ and $-\infty<a<b<\infty$. The fractional integration by parts formula
\begin{equation}\label{IP}
\int_{a}^{b}f(t)D^\alpha_{a|t}g(t)\,dt \;=\; \int_{a}^{b} 
g(t)D^\alpha_{t|b}f(t)\,dt,
\end{equation}
is valid for every $f\in I^\alpha_{t|b}(L^p(a,b))$, $g\in I^\alpha_{a|t}(L^q(a,b))$ such that $\frac{1}{p}+\frac{1}{q}\leq 1+\alpha$, $p,q>1$, where
$$I^\alpha_{a|t}(L^q(0,T)):=\left\{f= I^\alpha_{a|t}h,\,\, h\in L^q(a,b)\right\},$$
and
$$I^\alpha_{t|b}(L^p(a,b)):=\left\{f= I^\alpha_{t|b}h,\,\, h\in L^p(a,b)\right\}.$$
\end{proposition}
\begin{remark}
A simple sufficient condition for functions $f$ and $g$ to satisfy (\ref{IP}) is that $f,g\in C[a,b],$ such that
$D^\alpha_{t|b}f(t),D^\alpha_{a|t}g(t)$ exist at every point $t\in[a,b]$ and are continuous.
\end{remark}

\begin{proposition}\cite[Chapter~1]{SKM}\\
For $0<\alpha<1$, $-\infty<a<b<\infty$, we have the following identities
\begin{equation}\label{I3}
    D^\alpha_{a|t}I^\alpha_{a|t}f(t)=f(t),\,\hbox{a.e. $t\in(a,b)$}, \quad\hbox{for all}\,f\in L^r(a,b), 1\leq r\leq\infty,
\end{equation}
and
\begin{equation}\label{I4}
   -D.D^\alpha_{t|b}f=D^{1+\alpha}_{t|b}f,\quad\hbox{for all}\,f\in AC^2[a,b],
\end{equation}
where $D:=\frac{d}{dt}$.
\end{proposition}

Given $T>0$, let us define the functions $w_1$ and $w_2$ by
\begin{equation}\label{w1}
\displaystyle w_1(t)=\left(1-t/T\right)^\sigma,\quad\hbox{for all}\,\,\,0\leq t\leq T,
\end{equation}
and
\begin{equation}\label{w2}
\displaystyle w_2(t)=\left(1+t/T\right)^\sigma,\quad\hbox{for all}\,\,\,-T\leq t\leq 0,
\end{equation}
where $\sigma\gg1$ is bigg enough. Later on, we need the following properties concerning the functions $w_i$, $i=1,2$.
\begin{lemma}\cite[(2.45), p. 40]{SKM}\label{L1}\\
Let $T>0$, $0<\alpha<1$. For all $t\in[0,T]$, we have
\begin{equation}\label{P1}
D_{t|T}^\alpha
w_1(t)=\frac{\Gamma(\sigma+1)}{
\Gamma(\sigma+1-\alpha)}T^{-\alpha}(1-t/T)^{\sigma-\alpha},
\end{equation}
and
\begin{equation}\label{P3}
D_{t|T}^{1+\alpha}
w_1(t)=\frac{\Gamma(\sigma+1)}{
\Gamma(\sigma-\alpha)}T^{-(1+\alpha)}(1-t/T)^{\sigma-\alpha-1}.
\end{equation}
\end{lemma}
\begin{lemma}\cite[(2.45), p. 40]{SKM}\label{L2}\\
Let $T>0$, $0<\alpha<1$. For all $t\in[-T,0]$, we have
\begin{equation}\label{P2}
D_{t|0}^\alpha
w_2(t)=\frac{\Gamma(\sigma+1)}{
\Gamma(\sigma+1-\alpha)}T^{-\alpha}(1+t/T)^{\sigma-\alpha},
\end{equation}
and
\begin{equation}\label{P4}
D_{t|0}^{1+\alpha}
w_2(t)=\frac{\Gamma(\sigma+1)}{
\Gamma(\sigma-\alpha)}T^{-(1+\alpha)}(1+t/T)^{\sigma-\alpha-1}.
\end{equation}
\end{lemma}

\begin{lemma}\label{L4}
Let $T>0$, $0<\alpha<1$, $p>1$, we have
\begin{equation}\label{22}
\int_{0}^T (w_1(t))^{-1/(p-1)}|D_{t|T}^\alpha
w_1(t)|^{p/(p-1)}\,dt=C\,T^{1-\alpha p/(p-1)},
\end{equation}
and
\begin{equation}\label{23}
\int_0^T (w_1(t))^{-1/(p-1)}|D_{t|T}^{1+\alpha}
w_1(t)|^{p/(p-1)}\,dt=C\,T^{1-(1+\alpha) p/(p-1)},
\end{equation}
for some $C>0$.
\end{lemma}
\proof The proof of this lemma can be found, e.g., in Furati \& Kirane \cite{Furati}. To make this paper self-contained, we will present the proof in details. First, we prove \eqref{22}, while the identity \eqref{23} can be done similarly. Using Lemma \ref{L1}, we have
\begin{eqnarray*}
\int_{0}^T (w_1(t))^{-1/(p-1)}|D_{t|T}^\alpha
w_1(t)|^{p/(p-1)}\,dt&=&C\,T^{-\alpha}\int_{0}^T (w_1(t))^{-1/(p-1)}(w_1(t))^{\frac{p(\sigma-\alpha)}{(p-1)\sigma}}\,dt\\
&=&C\,T^{-\alpha\frac{ p}{p-1}}\int_{0}^T(1-t/T)^{\sigma-\alpha\frac{ p}{p-1}}\,dt\\
&=&C\,T^{1-\alpha\frac{ p}{p-1}}\int_{0}^1(1-s)^{\sigma-\alpha\frac{ p}{p-1}}\,ds\\
&=&C\,T^{1-\alpha\frac{ p}{p-1}}.
\end{eqnarray*}
\hfill$\square$\\
Similarly, we have
\begin{lemma}\label{L5}
Let $T>0$, $0<\alpha<1$, $p>1$, we have
\begin{equation}
\int_{-T}^0 (w_2(t))^{-1/(p-1)}|D_{t|0}^\alpha
w_2(t)|^{p/(p-1)}\,dt=C\,T^{1-\alpha p/(p-1)},
\end{equation}
and
\begin{equation}
\int_{-T}^0 (w_2(t))^{-1/(p-1)}|D_{t|0}^{1+\alpha}
w_2(t)|^{p/(p-1)}\,dt=C\,T^{1-(1+\alpha) p/(p-1)},
\end{equation}
for some $C>0$.
\end{lemma}
On the other hand, 
\begin{lemma}\cite[Lemma~8.18]{Folland}($C^\infty$ Urysohn Lemma)\label{Urysohn}\\
If $K\subset\mathbb{R}^n$ is compact and $U$ is an open set containing $K$, there exists $f\in C^\infty_c(\mathbb{R}^n)$ such that $0\leq f\leq 1$, $f=1$ on $K$, and supp$f\subset U$.
\end{lemma}
Using Lemma \ref{Urysohn}, there exists a function $\phi\in C^\infty_c(\mathbb{R}^n)$ such that 
\begin{equation}\label{testfunction}
\phi(x)=\phi(|x|)=\left\{
\begin{array}{ll}
1&\hbox{if}\;\;|x|\leq1,\\
\\
0&\hbox{if}\;\;|x|\geq2.
\end{array}
\right.
\end{equation}
An explicit example of this function can be found in \cite[Chapter~1, p.40]{Yuta}.
\begin{lemma}\label{lemma4}
Let $m\geq1$, $\ell>2m$, and $\phi$ is defined in \eqref{testfunction}. Then, the following estimate holds:
$$|\Delta^m(\phi^\ell)| \leq C\,\phi^{\ell-2m},$$
for some $C=C(m,\ell) >0$.
\end{lemma}
\begin{proof}
First, we recall the following formula of derivatives of composed functions for $|\alpha|\ge 1$:
$$ \partial_x^\alpha h\big(f(x)\big)= \sum_{k=1}^{|\alpha|}h^{(k)} \big(f(x)\big)\left(\sum_{\substack{\gamma_1+\cdots+\gamma_k \le \alpha\\ |\gamma_1|+\cdots+|\gamma_k|= |\alpha|,\, |\gamma_i|\ge 1}}\big(\partial_x^{\gamma_1} f(x)\big) \cdots \big(\partial_x^{\gamma_k} f(x)\big)\right), $$
where $h=h(z)$ and $h^{(k)}(z)=\frac{d^k h(z)}{dz^k}$. Applying this formula with $h(z)= z^\ell$ and $f(x)= \phi(x)$, $1\leq|x|\leq 2$, we obtain
$$
\big|\partial_x^\alpha  (\phi(x))^\ell\big|\le \sum_{k=1}^{|\alpha|}\ell(\ell-1)\dots(\ell-k+1) (\phi(x))^{\ell-k}\left(\sum_{\substack{\gamma_1+\cdots+\gamma_k \le \alpha\\ |\gamma_1|+\cdots+|\gamma_k|= |\alpha|,\, |\gamma_i|\ge 1}}\big|\partial_x^{\gamma_1} \phi(x)\big| \cdots \big|\partial_x^{\gamma_k} \phi(x)\big|\right) 
$$
Using $\phi\in C^\infty_c(\mathbb{R}^n)$, we have
$$\big|\partial_x^{\gamma_i} \phi(x)\big| \leq C_\alpha,\qquad \hbox{for all}\,1\leq i\leq k,\,1\leq|x|\leq 2,$$
for some constant $C_\alpha=C(\alpha)>0$, which implies,
$$
\big|\partial_x^\alpha  (\phi(x))^\ell\big|\le \tilde{C}_\alpha\sum_{k=1}^{|\alpha|}\ell(\ell-1)\dots(\ell-k+1) (\phi(x))^{\ell-k},
$$
for some constant $\tilde{C}_\alpha>0$. Therefore, as $\phi\leq1$ and $\ell-k\geq \ell-|\alpha|$ for all $1\leq k\leq|\alpha|$, we conclude that
$$
\big|\partial_x^\alpha  (\phi(x))^\ell\big|\le C_{\alpha,\ell} (\phi(x))^{\ell-|\alpha|},\qquad\hbox{for all}\,\,1\leq|x|\leq 2
$$
Finally, as
$$|\Delta^m(\phi^\ell)| \leq m\sum_{|\alpha|=m}\big|\partial_x^{2\alpha}  (\phi(x))^\ell\big|,$$
the proof is complete.
\end{proof}

\begin{lemma}\label{L3}
Let $m\geq1$, $R>0$, $\ell>2mp/(p-1)$, and $p>1$. Then, the following estimate holds
$$\int_{\mathbb{R}^n}(\phi_R(x))^{-\frac{1}{p-1}}\,\big|(-\Delta)^m\phi_R(x)\big|^{\frac{p}{p-1}}\, dx\leq C R^{-\frac{2mp}{p-1}+n},$$
for some $C>0$, where $\phi_R(x):= \phi^\ell({x}/{R})$ and $\phi$ is given in \eqref{testfunction}.
\end{lemma}
\begin{proof} Using the change of variables $\tilde{x}=x/R$, we have 
$$(-\Delta)^m\phi_R(x)=R^{-2m}(-\Delta)^m\phi(\tilde{x}).$$ 
Therefore, by Lemma \ref{lemma4}, we conclude that
$$\int_{\mathbb{R}^n}(\phi_R(x))^{-\frac{1}{p-1}}\,\big|(-\Delta)^m \phi_R(x)\big|^{\frac{p}{p-1}}\, dx\leq C R^{-\frac{2mp}{p-1}+n}\int_{|\tilde{x}|\leq 2}(\phi(\tilde{x}))^{\ell-2m\frac{p}{p-1}}\, d \tilde{x}\leq C R^{-\frac{2mp}{p-1}+n}.$$
\end{proof}


\section{Local existence}\label{section3}

This section is dedicated to proving the local existence and uniqueness of mild solutions to the problem \eqref{1}. Let us start by the
\begin{definition}[Mild solution]
Let $u_0\in C_0(\mathbb{R}^n)$, $0<\gamma< 1$, $m\geq 1$, $p>1$ and $T>0.$ We say that $u\in C([0,T],C_0(\mathbb{R}^n))$ is a mild solution of problem \eqref{1} if $u$ satisfies the following integral equation
\begin{equation}\label{IE}
    u(t)=S(t)u_0+C_\alpha \int_0^tS(t-s)I_{0|s}^\alpha(|u(s)|^p)\,ds,\quad t\in[0,T],
\end{equation}
where $\alpha:=1-\gamma\in(0,1)$ and $C_\alpha=\Gamma(\alpha)$.
\end{definition}
\noindent{\bf Proof of Theorem \ref{T0+}.} For arbitrary $T>0,$ let
\[E_T:=\left\{u\in C([0,T],C_0(\mathbb{R}^n));\;
\|u(t)\|_\infty\leq 2\|u_0\|_\infty, \,\text{for all}\,t\in[0,T]\right\},
\]
where
$\|\cdotp\|_\infty:=\|\cdotp\|_{L^\infty(\mathbb{R}^n)}$, and we equip $E_T$ with the following metric generated by the norm of $C([0,T],C_0(\mathbb{R}^n))$ 
$$d(u,v)=\max_{t\in[0,T]}\|u(t)-v(t)\|_\infty,\quad\text{for all}\,\,u,v\in E_T.$$
Since $C([0,T],C_0(\mathbb{R}^n))$ is a Banach space, $(E_T,d)$ is a complete metric space. Next, for
all $u\in E_T,$ we define
\[
\Psi(u)(t):=S(t)u_0+C_\alpha \int_0^tS(t-s)I_{0|s}^\alpha(|u(s)|^p)\,ds.
\]
We prove the local existence by the Banach fixed point
theorem.\\
\noindent$\bullet$ {\bf${\bf\Psi:E_T\rightarrow E_T}$:}$\;$Let $u\in E_T,$ using (\ref{LpLq}), we obtain
\begin{eqnarray*}
\|\Psi(u)(t)\|_\infty &\leq& \|u_0\|_\infty+  \int_0^t\int_0^s(s-\sigma)^{-\gamma}\|u(\sigma)\|^p_\infty\,d\sigma\,ds\\
   &=& \|u_0\|_\infty+   \int_0^t\int_\sigma^t(s-\sigma)^{-\gamma}\|u(\sigma)\|^p_\infty\,ds\,d\sigma \\
     &\leq&\|u_0\|_\infty+\frac{T^{2-\gamma}2^p\|u_0\|_{L^\infty}^{p-1}}
   {(1-\gamma)(2-\gamma)}\|u_0\|_\infty,
\end{eqnarray*}
for all $t\in[0,T]$. Now, if we choose $T$ small enough such that
\begin{equation}\label{conditionssurT+}
\frac{T^{2-\gamma}2^p\|u_0\|_{\infty}^{p-1}}{(1-\gamma)(2-\gamma)}\leq1,
\end{equation}
we conclude that $\|\Psi(u)(t)\|_\infty\leq
2\|u_0\|_{\infty}$, for all $t\in[0,T]$. Therefore, using the fact that $S(t): C_0(\mathbb{R}^n)  \longrightarrow C_0(\mathbb{R}^n)$, and the continuity of the semigroup $S(t)$, we get $\Psi(u)\in E_T.$\\

\noindent$\bullet$ {\bf$\Psi$ is a contraction:}$\;$ For $u,v\in E_T$, using again (\ref{LpLq}), we have
\begin{eqnarray*}
 \|\Psi(u)(t)-\Psi(v)(t)\|_\infty&\leq& \int_0^t\int_0^s(s-\sigma)^{-\gamma}
  \||u(\sigma)|^p-|v(\sigma)|^p\|_{\infty}\,d\sigma\,ds\\
   &=&   \int_0^t\int_\sigma^t(s-\sigma)^{-\gamma}\||u(\sigma)|^p-|v(\sigma)|^p\|_{\infty}\,ds\,d\sigma \\
   &\leq& \frac{C(p)2^p\|u_0\|_{\infty}^{p-1}T^{2-\gamma}}
   {(1-\gamma)(2-\gamma)}d(u,v) \\
   &\leq& \frac{1}{2}d(u,v),
\end{eqnarray*}
for all $t\in[0,T]$, thanks to the following inequality
\begin{equation}\label{estimationimp}
||u|^p-|v|^p|\leq C(p)|u-v|(|u|^{p-1}+|v|^{p-1});
\end{equation}
$T$ is chosen such that
\begin{equation}\label{esti5+}
\frac{T^{2-\gamma}2^p\|u_0\|_{\infty}^{p-1}\max(2C(p),1)}{(1-\gamma)(2-\gamma)}\leq1.
\end{equation}
Then, by the Banach fixed point theorem, see e.g. \cite[Theorem~1.1.1]{CH}, there exists a mild solution $u\in \Pi_T:=L^\infty((0,T),C_0(\mathbb{R}^N)),$
 to problem (\ref{1}).\\

\noindent$\bullet$ {\bf Uniqueness:}$\;$ If $u,v$ are two mild solutions in $E_T$ for some $T>0,$
using (\ref{LpLq}) and  (\ref{estimationimp}), we obtain
\begin{eqnarray*}
  \|u(t)-v(t)\|_{\infty}&\leq&C(p)2^p\|u_0\|_{\infty}^{p-1}  \int_0^t\int_0^s(s-\sigma)^{-\gamma}
  \|u(\sigma)-v(\sigma)\|_{\infty}\,d\sigma\,ds\\
   &=&C(p)2^p\|u_0\|_{\infty}^{p-1} 
   \int_0^t\int_\sigma^t(s-\sigma)^{-\gamma}\|u(\sigma)-
   v(\sigma)\|_{\infty}\,ds\,d\sigma \\
   &=&\frac{C(p)2^p\|u_0\|_{\infty}^{p-1}}
   {1-\gamma} \int_0^t(t-\sigma)^{1-\gamma}\|u(\sigma)-
   v(\sigma)\|_{\infty}\,d\sigma,
\end{eqnarray*}
for all $t\in[0,T]$. So the uniqueness follows from Gronwall's inequality (cf. \cite{CH}).

Next, using the uniqueness of solutions, we conclude the existence
of a maximal solution 
\[
u\in C([0,T_{\max}),C_0(\mathbb{R}^n)).
\]
where
\[
T_{\max}:=\sup\left\{T>0\;;\;\text{there exist a mild solution $u\in E_T$
to (\ref{1})}\right\}\leq+\infty.
\]
Moreover, if $0\leq t\leq t+\tau< T_{\max},$ using (\ref{IE}), we can write
\begin{eqnarray}\label{newIE+}
u(t+\tau)&=&S(\tau)u(t)+C_\alpha\int_0^\tau S(\tau-s)\int_0^s(s-\sigma)^{-\gamma}|u(t+\sigma)|^p\,d\sigma\,ds\nonumber\\
&{}&+\,C_\alpha\int_0^\tau S(\tau-s)\int_0^t(t+s-\sigma)^{-\gamma}|u(\sigma)|^p\,d\sigma\,ds.
\end{eqnarray}
To prove that $\|u(t)\|_\infty\rightarrow\infty$ as $t\rightarrow T_{\max},$
whenever $T_{\max}<\infty,$ we proceed by contradiction. Suppose that $u$ is a solution of
(\ref{IE}) on some interval $[0,T)$ with $\|u\|_{L^\infty((0,T)\times\mathbb{R}^n)}<\infty$ and $T_{\max}<\infty.$
Using the fact that the last term in $(\ref{newIE+})$ depends only on the values of $u$ in the interval
$(0,t)$ and using again a fixed-point argument, we conclude that $u$ can be extended to a solution
on some interval $[0,T')$ with $T'>T.$ If we repeat this iteration, we obtain a contradiction with
the fact that the maximal time $T_{\max}$ is finite.\\

\noindent$\bullet$ {\bf Regularity:}$\;$ If $u_0\in
L^r(\mathbb{R}^n)\cap C_0(\mathbb{R}^n),$ for $1\leq r<\infty,$ then by repeating the
fixed point argument in the metric space
\begin{eqnarray*}
  E_{T,r}&:=& \left\{u\in C([0,T],C_0(\mathbb{R}^n)
  \cap L^r(\mathbb{R}^n));\right.\\
 &{}&\qquad\qquad \left.\;\|u(t)\|_\infty\leq 2\|u_0\|_{L^\infty},\|u(t)\|_{r}\leq
   2\|u_0\|_{L^r}, \text{for all}\,\,t\in[0,T]\right\},
\end{eqnarray*}
equipped with
$$d_r(u,v)=\max_{t\in[0,T]}\left(\|u(t)-v(t)\|_\infty+\|u(t)-v(t)\|_r\right),\quad\text{for all}\,\,u,v\in E_{T,r},$$
instead of $(E_T,d)$, where $\|\cdotp\|_{r}:=\|\cdotp\|_{L^r(\mathbb{R}^n)},$
and by estimating $\|u^p\|_{L^r(\mathbb{R}^n)}$ by
$\|u\|^{p-1}_{L^\infty(\mathbb{R}^n)}\|u\|_{L^r(\mathbb{R}^n)}$ in
the contraction mapping argument, using (\ref{LpLq}), we obtain a
unique solution in $E_{T,r}$, and therefore we conclude that
\[
u\in
C([0,T_{\max}),C_0(\mathbb{R}^n)\cap L^r(\mathbb{R}^n)).
\]
\hfill$\square$

\begin{remark} If $T_{\max}=\infty$, the solution $u$ is said to be global in time, while $u$ is said to blow up in a
finite time when $T_{\max}<\infty,$ and in this case we have $\|u(t)\|_{L^\infty(\mathbb{R}^N)}\rightarrow\infty$ as $t\rightarrow T_{\max}.$
\end{remark}


\section{Blow-up Rate}\label{section4}
\setcounter{equation}{0}
 In this section, we prove the blow-up rate for the blowing-up solutions of problem \eqref{1}, namely Theorem \ref{Theorem1}. We take the solution of (\ref{1}) with an initial condition
satisfying
\begin{equation}\label{13}
    u_0\in L^1(\mathbb{R}^n)\cap C_0(\mathbb{R}^n),\quad \int_{\mathbb{R}^n}u_0(x)\,dx>0.
\end{equation}
The following Lemma will be used in the proof of Theorem
\ref{Theorem1}. The test function method (see
\cite{Baras,BarasPierre,KQ, PM1, Zhang} and
the references therein) is the key to prove this Lemma.
\begin{lemma}\label{lemma1+}
Let $v$ be a bounded classical solution of
\begin{equation}\label{2+}
v_t+(-\Delta)^m v=\int_{-\infty}^t(t-s)^{-\gamma}|v(s)|^p\,ds \quad \mbox{in}\;\mathbb{R}\times\mathbb{R}^n,
\end{equation}
$m\geq1$, $p>1$. Then $v\equiv0$ whenever
\begin{equation}
p\leq 1+
2m(2-\gamma)/(N-2m+2m\gamma)_+\qquad\hbox{or}\qquad p<1/\gamma.
\end{equation}
\end{lemma}
\begin{proof}
Suppose that $v$ is a bounded classical solution to \eqref{2+} on $\mathbb{R}^n\times\mathbb{R}$. Then 
\begin{equation}\label{5}
v_t+(-\Delta)^m v=\int_{-\infty}^t(t-s)^{-\gamma}|v(s)|^p\,ds \quad \text{in}\,[0,\infty)\times\mathbb{R}^n,
\end{equation}
and
\begin{equation}\label{9}
v_t+(-\Delta)^m v=\int_{-\infty}^t(t-s)^{-\gamma}|v(s)|^p\,ds \quad\text{in}\,(-\infty,0]\times\mathbb{R}^n.
\end{equation}
Let $w_i$, $i=1,2$, and $\phi$ be the functions defined, respectively, in \eqref{w1}-\eqref{w2} and \eqref{testfunction}. For $T,R\gg1$, let us define our test functions as follows:
$$ \varphi_1(t,x)=D^\alpha_{t|T}\left(\tilde{\varphi}_1(t,x)\right),\qquad (t,x)\in[0,T]\times\mathbb{R}^n,$$
and
$$ \varphi_2(t,x)=D^\alpha_{t|0}\left(\tilde{\varphi}_2(t,x)\right),\qquad (t,x)\in[-T,0]\times\mathbb{R}^n,$$
 where $\alpha=1-\gamma$, $\tilde{\varphi}_i(t,x)=w_i(t)\phi_R(x)$, $i=1,2$, and $\phi_R(x)=\phi^\ell(x/R)$, $\ell>2mp/(p-1)$.\\
Multiplying \eqref{5} by $\varphi_1(t,x)$ (resp.  \eqref{9} by $\varphi_2(t,x)$) and integrating over $[0,T]\times\mathbb{R}^n$ (resp. over $[-T,0]\times\mathbb{R}^n$), we get
\begin{eqnarray*}
\int_0^T\int_{\mathbb{R}^n}\int_{-\infty}^t(t-s)^{-\gamma}|v(s)|^p\,ds\varphi_1(t,x) \,dx\,dt&=&\int_0^T\int_{\mathbb{R}^n}v(t,x)(-\Delta)^{m}\varphi_1(t,x)\,dx\,dt\\
&{}&+\int_0^T\int_{\mathbb{R}^n}v_t(t,x)\varphi_1(t,x)\,dx\,dt,
\end{eqnarray*}
and
\begin{eqnarray*}
\int_{-T}^0\int_{\mathbb{R}^n}\int_{-\infty}^t(t-s)^{-\gamma}|v(s)|^p\,ds\varphi_2(t,x) \,dx\,dt&=&\int_{-T}^0\int_{\mathbb{R}^n}v(t,x)(-\Delta)^{m}\varphi_2(t,x)\,dx\,dt\\
&{}&+\int_{-T}^0\int_{\mathbb{R}^n}v_t(t,x)\varphi_2(t,x)\,dx\,dt,
\end{eqnarray*}
where we have used the Green's identity several times. Using integration by parts, \eqref{P1}, and \eqref{P2}, we have
\begin{eqnarray*}
&{}&\int_0^T\int_{\mathbb{R}^n}\int_{-\infty}^t(t-s)^{-\gamma}|v(s)|^p\,ds\varphi_1(t,x) \,dx\,dt+C_{\alpha,\sigma}\,T^{-\alpha}\int_{\mathbb{R}^n}v(0,x)\phi_R(x)\,dx\\
&{}&=\int_0^T\int_{\mathbb{R}^n}v(t,x)(-\Delta)^{m}\varphi_1(t,x)\,dx\,dt-\int_0^T\int_{\mathbb{R}^n}v(t,x)\partial_t\varphi_1(t,x)\,dx\,dt,
\end{eqnarray*}
and
\begin{eqnarray*}
&{}&\int_{-T}^0\int_{\mathbb{R}^n}\int_{-\infty}^t(t-s)^{-\gamma}|v(s)|^p\,ds\varphi_2(t,x) \,dx\,dt-C_{\alpha,\sigma}\,T^{-\alpha}\int_{\mathbb{R}^n}v(0,x)\phi_R(x)\,dx\\
&{}&=\int_{-T}^0\int_{\mathbb{R}^n}v(t,x)(-\Delta)^{m}\varphi_2(t,x)\,dx\,dt-\int_{-T}^0\int_{\mathbb{R}^n}v(t,x)\partial_t\varphi_2(t,x)\,dx\,dt,
\end{eqnarray*}
where $C_{\alpha,\sigma}=\Gamma(\sigma+1)/\Gamma(\sigma+1-\alpha)$. Using again \eqref{P1}, and \eqref{P2}, we can see that $\varphi_i\geq0$, $i=1,2$, then
$$\int_0^T\int_{\mathbb{R}^n}\int_{0}^t(t-s)^{-\gamma}|v(s)|^p\,ds\varphi_1(t,x) \,dx\,dt\leq\int_0^T\int_{\mathbb{R}^n}\int_{-\infty}^t(t-s)^{-\gamma}|v(s)|^p\,ds\varphi_1(t,x) \,dx\,dt,$$
and
$$\int_{-T}^0\int_{\mathbb{R}^n}\int_{-T}^t(t-s)^{-\gamma}|v(s)|^p\,ds\varphi_2(t,x) \,dx\,dt\leq\int_{-T}^0\int_{\mathbb{R}^n}\int_{-\infty}^t(t-s)^{-\gamma}|v(s)|^p\,ds\varphi_2(t,x) \,dx\,dt,$$
which implies
\begin{eqnarray*}
&{}&\int_0^T\int_{\mathbb{R}^n}\int_{0}^t(t-s)^{-\gamma}|v(s)|^p\,ds\varphi_1(t,x) \,dx\,dt+C_{\alpha,\sigma}\,T^{-\alpha}\int_{\mathbb{R}^n}v(0,x)\phi_R(x)\,dx\\
&{}&\leq\int_0^T\int_{\mathbb{R}^n}v(t,x)(-\Delta)^{m}\varphi_1(t,x)\,dx\,dt-\int_0^T\int_{\mathbb{R}^n}v(t,x)\partial_t\varphi_1(t,x)\,dx\,dt,
\end{eqnarray*}
and
\begin{eqnarray*}
&{}&\int_{-T}^0\int_{\mathbb{R}^n}\int_{-T}^t(t-s)^{-\gamma}|v(s)|^p\,ds\varphi_2(t,x) \,dx\,dt-C_{\alpha,\sigma}\,T^{-\alpha}\int_{\mathbb{R}^n}v(0,x)\phi_R(x)\,dx\\
&{}&\leq\int_{-T}^0\int_{\mathbb{R}^n}v(t,x)(-\Delta)^{m}\varphi_2(t,x)\,dx\,dt-\int_{-T}^0\int_{\mathbb{R}^n}v(t,x)\partial_t\varphi_2(t,x)\,dx\,dt,
\end{eqnarray*}
that is
\begin{eqnarray}\label{10}
&{}&\Gamma(\alpha)\int_0^T\int_{\mathbb{R}^n}I^\alpha_{0|t}(|v|^p)D^\alpha_{t|T}\tilde{\varphi}_1(t,x) \,dx\,dt+C_{\alpha,\sigma}\,T^{-\alpha}\int_{\mathbb{R}^n}v(0,x)\phi_R(x)\,dx\nonumber\\
&{}&\leq\int_0^T\int_{\mathbb{R}^n}v(t,x)(-\Delta)^{m}\varphi_1(t,x)\,dx\,dt-\int_0^T\int_{\mathbb{R}^n}v(t,x)\partial_t\varphi_1(t,x)\,dx\,dt,
\end{eqnarray}
and
\begin{eqnarray}\label{11}
&{}&\Gamma(\alpha)\int_{-T}^0\int_{\mathbb{R}^n}I^\alpha_{-T|t}(|v|^p)D^\alpha_{t|0}\tilde{\varphi}_2(t,x) \,dx\,dt-C_{\alpha,\sigma}\,T^{-\alpha}\int_{\mathbb{R}^n}v(0,x)\phi_R(x)\,dx\nonumber\\
&{}&\leq\int_{-T}^0\int_{\mathbb{R}^n}v(t,x)(-\Delta)^{m}\varphi_2(t,x)\,dx\,dt-\int_{-T}^0\int_{\mathbb{R}^n}v(t,x)\partial_t\varphi_2(t,x)\,dx\,dt,
\end{eqnarray}
where $I^\alpha_{0|t}$ and $I^\alpha_{-T|t}$ are defined in \eqref{I1}. Adding \eqref{10} with \eqref{11}, and using \eqref{IP}-\eqref{I3} , we may obtain
\begin{eqnarray}\label{12}
&{}&\Gamma(\alpha)I(v)+\Gamma(\alpha)J(v)\nonumber\\
&{}&\leq \int_0^T\int_{\mathbb{R}^n}v(t,x)(-\Delta)^{m}\varphi_1(t,x)\,dx\,dt-\int_0^T\int_{\mathbb{R}^n}v(t,x)\partial_t\varphi_1(t,x)\,dx\,dt\nonumber\\
&{}&\quad+\,\int_{-T}^0\int_{\mathbb{R}^n}v(t,x)(-\Delta)^{m}\varphi_2(t,x)\,dx\,dt-\int_{-T}^0\int_{\mathbb{R}^n}v(t,x)\partial_t\varphi_2(t,x)\,dx\,dt,
\end{eqnarray}
where
$$I(v)=\int_0^T\int_{\mathbb{R}^n}|v(t,x)|^p\tilde{\varphi}_1(t,x) \,dx\,dt\quad\text{and}\quad J(v)=\int_{-T}^0\int_{\mathbb{R}^n}|v(t,x)|^p\tilde{\varphi}_2(t,x) \,dx\,dt.$$
Using \eqref{I4}, we get
\begin{eqnarray}\label{4}
\Gamma(\alpha)I(v)+\Gamma(\alpha)J(v)&\leq&\int_0^T\int_{\mathbb{R}^n}|v(t,x)|D^{\alpha}_{t|T}w_1(t)|\Delta^{m}\phi_R(x)|\,dx\,dt\nonumber\\
&{}&\,+\int_0^T\int_{\mathbb{R}^n}|v(t,x)|\phi_R(x)|D^{1+\alpha}_{t|T}w_1(t)|\,dx\,dt\nonumber\\
&{}&\,+\int_{-T}^0\int_{\mathbb{R}^n}|v(t,x)|D^{\alpha}_{t|0}w_2(t)|\Delta^{m}\phi_R(x)|\,dx\,dt\nonumber\\
&{}&\,+\int_{-T}^0\int_{\mathbb{R}^n}|v(t,x)|\phi_R(x)|D^{1+\alpha}_{t|0}w_2(t)|\,dx\,dt\nonumber\\
&=:&I_1+I_2+J_1+J_2.
\end{eqnarray}
We start to estimate $I_1$. Using H\"older's estimate, we have
\begin{eqnarray*}
I_1&=&\int_0^T\int_{|x|>R}|v(t,x)|\tilde{\varphi}^{1/p}\tilde{\varphi}^{-1/p}D^{\alpha}_{t|T}w_1(t)|\Delta^{m}\phi_R(x)|\,dx\,dt\\
&\leq&\left(\tilde{I}(v)\right)^{1/p}\left(\int_0^T\int_{\mathbb{R}^n}(\tilde{\varphi}(t,x))^{-\frac{1}{p-1}}(D^{\alpha}_{t|T}w_1(t))^{\frac{p}{p-1}}|\Delta^{m}\phi_R(x)|^{\frac{p}{p-1}}\,dx\,dt\right)^{\frac{p-1}{p}}\\
&=&\left(\tilde{I}(v)\right)^{\frac{1}{p}}\left(\int_0^T(w_1(t))^{-\frac{1}{p-1}}(D^{\alpha}_{t|T}w_1(t))^{\frac{p}{p-1}}\,dt\int_{\mathbb{R}^n}(\phi_R(x))^{-\frac{1}{p-1}}|\Delta^{m}\phi_R(x)|^{\frac{p}{p-1}}\,dx\right)^{\frac{p-1}{p}}
\end{eqnarray*}
where
$$\tilde{I}(v)=\int_0^T\int_{|x|>R}|v(t,x)|^p\tilde{\varphi}(t,x)\,dx\,dt.$$
Using Lemmas \ref{L4} and \ref{L3}, we obtain
\begin{equation}\label{2}
I_1\leq C\,\left(\tilde{I}(v)\right)^{\frac{1}{p}} T^{\frac{p-1}{p}-\alpha}\,R^{\frac{n(p-1)}{p}-2m}.
\end{equation}
Similar, we estimate $I_2$ as follows:
\begin{eqnarray*}
I_2&=&\int_0^T\int_{|x|\leq 2R}|v(t,x)|\tilde{\varphi}^{1/p}\tilde{\varphi}^{-1/p}\phi_R(x)|D^{1+\alpha}_{t|T}(w_1(t))|\,dx\,dt\\
&\leq&\left(I(v)\right)^{\frac{1}{p}}\left(\int_0^T\int_{|x|\leq 2R}(w_1(t))^{-\frac{1}{p-1}}|D^{1+\alpha}_{t|T}w_1(t)|^{\frac{p}{p-1}}\phi_R(x)\,dx\,dt\right)^{\frac{p-1}{p}}\\
&=&\left(I(v)\right)^{\frac{1}{p}}\left(\int_0^T(w_1(t))^{-\frac{1}{p-1}}|D^{1+\alpha}_{t|T}w_1(t)|^{\frac{p}{p-1}}\,dt\int_{|x|\leq 2R}\phi_R(x)\,dx\right)^{\frac{p-1}{p}}.
\end{eqnarray*}
By the change of variable: $\tilde{x}=x/R$, we have
$$\int_{|x|\leq 2R}\phi_R(x)\,dx=\int_{|\tilde{x}|\leq 2}\phi(\tilde{x})R^n\,d\tilde{x}=C\,R^n.$$
Therefore, using Lemma \ref{L4}, we conclude that
$$
I_2\leq C\,\left(I(v)\right)^{\frac{1}{p}} T^{\frac{p-1}{p}-1-\alpha}\,R^{\frac{n(p-1)}{p}}.
$$
By $\varepsilon$-Young's inequality 
$$
ab\leq\;\varepsilon a^p+C_\varepsilon\,b^{\frac{p}{p-1}}\qquad\text{where}\;\; a>0,b>0,\;\;p>1,$$
the following estimation holds
\begin{equation}\label{3}
I_2\leq \varepsilon\,I(v)+C\,T^{1-(1+\alpha)\frac{p}{p-1}}\,R^{n}.
\end{equation}
Similarly, using Lemma \ref{L5} instead of Lemma \ref{L4}, we get
\begin{equation}\label{14}
J_1\leq C\,\left(\tilde{J}(v)\right)^{\frac{1}{p}} T^{\frac{p-1}{p}-\alpha}\,R^{\frac{n(p-1)}{p}-2m},
\end{equation}
and
\begin{equation}\label{15}
J_2\leq \varepsilon\,J(v)+C\,T^{1-(1+\alpha)\frac{p}{p-1}}\,R^{n}.
\end{equation}
Insert \eqref{2}, \eqref{3}, \eqref{14} and \eqref{15} in \eqref{4}, and choose $\varepsilon<\Gamma(\alpha)$,  we get
\begin{eqnarray}\label{8}
I(v)+J(v)&\leq& C\,T^{1-(1+\alpha)\frac{p}{p-1}}\,R^{n}+ C\,\left(\tilde{I}(v)\right)^{\frac{1}{p}} T^{\frac{p-1}{p}-\alpha}\,R^{\frac{n(p-1)}{p}-2m}\nonumber\\
&{}&+\, C\,\left(\tilde{J}(v)\right)^{\frac{1}{p}} T^{\frac{p-1}{p}-\alpha}\,R^{\frac{n(p-1)}{p}-2m}.
\end{eqnarray}
 At this stage, we have to distinguish three cases:\\

\noindent $\bullet$ The case $p<p_\gamma$. Take $R=T^{\frac{1}{2m}}$, we obtain
$$I(v)\leq C\,T^{-\delta}+C\,\left(\tilde{I}(v)\right)^{\frac{1}{p}} T^{-\delta\frac{(p-1)}{p}}.$$
where $\delta=-1+(1+\alpha)\frac{p}{p-1}-\frac{n}{2m}$. Using the fact that $\tilde{I}(v)\leq I(v)$, $\varepsilon$-Young's inequality, we infer that
$$I(v)+J(v)\leq C\,T^{-\delta}+\frac{1}{2}I(v)+\frac{1}{2}J(v),$$
i.e.
\begin{equation}\label{6}
I(v)+J(v)\leq C\,T^{-\delta}.
\end{equation}
As $p<p_\gamma$ implies $\delta>0$, after passing to the limit as $T\rightarrow\infty$, using the monotone convergence theorem, the continuity of $v$ in time and space, we get 
$$\int_{-\infty}^\infty\int_{\mathbb{R}^n}|v(t,x)|^p \,dx\,dt=0.$$ 
Therefore $v=0$ in $\mathbb{R}\times\mathbb{R}^n$.\\
\noindent $\bullet$ The case $p=p_\gamma$. On the one hand, from \eqref{6}, we have
$$v\in L^p((0,\infty),L^p(\mathbb{R}^n))\quad\text{and}\quad v\in L^p((-\infty,0),L^p(\mathbb{R}^n))$$
which implies that
\begin{equation}\label{7}
\tilde{I}(v), \tilde{J}(v)\longrightarrow 0,\quad\hbox{as}\,\, T\rightarrow\infty.
\end{equation}
On the other hand, take $R=T^{\frac{1}{2m}}K^{-\frac{1}{2m}}$, where $1\leq K<T$ is large enough such that when $T\rightarrow\infty$ we don't have $K\rightarrow\infty$ at the same time. From \eqref{8} and $p=p_\gamma$, we obtain
$$
I(v)+J(v)\leq C\,K^{-\frac{n}{2m}}+ C\,\left(\tilde{I}(v)\right)^{\frac{1}{p}} \,K^{-\frac{n(p-1)}{2mp}+2}+ C\,\left(\tilde{J}(v)\right)^{\frac{1}{p}} \,K^{-\frac{n(p-1)}{2mp}+2}.
$$
Letting the limit as $T\rightarrow\infty$, using \eqref{7}, and the monotone convergence theorem, we get
$$
\int_{-\infty}^\infty\int_{\mathbb{R}^n}|v(t,x)|^p \,dx\,dt\leq C\,K^{-\frac{n}{2m}}.
$$
Taking the limit as $K\rightarrow\infty$, we conclude as above that $v=0$ in $\mathbb{R}\times\mathbb{R}^n$.\\
\noindent $\bullet$The case $p< 1 / \gamma$. In this case, we choose $R\in[1,T)$ large enough such that when $T\rightarrow\infty$ we don't have $R\rightarrow\infty$ at the same time. From \eqref{8}, and the fact that $\tilde{I}(v)\leq I(v)$, $\tilde{J}(v)\leq J(v)$, $\varepsilon$-Young's inequality, we infer that
$$I(v)+J(v)\leq C\,T^{1-(1+\alpha)\frac{p}{p-1}}\,R^{n}+ \frac{1}{2}I(v)+ \frac{1}{2}J(v)+C\, T^{1-\alpha\frac{p}{p-1}}\,R^{n-2m\frac{p}{p-1}},$$
i.e.
$$I(v)+J(v)\leq C\,T^{1-(1+\alpha)\frac{p}{p-1}}\,R^{n}+C\, T^{1-\alpha\frac{p}{p-1}}\,R^{n-2m\frac{p}{p-1}},$$
Letting the limit as $T\rightarrow\infty$, and the fact that $p< 1 / \gamma\Rightarrow 1-\alpha\frac{p}{p-1}<0$, we get
$$\int_{-\infty}^\infty\int_{\mathbb{R}^n}|v(t,x)|^p \,dx\,dt=0.$$ 
Therefore $v=0$ in $\mathbb{R}\times\mathbb{R}^n$.\\
This completes the proof.
\end{proof}

\noindent{\bf Proof of Theorem \ref{Theorem1}.} The proof is in two parts:\\

\noindent $\bullet$ \underline{The upper blow-up rate estimate}. Let
$$M(t):=\sup_{\mathbb{R}^n\times (0,t]}|u|,\qquad t\in(0,T^*).$$
Clearly, $M$ is positive, continuous, nondecreasing in $(0,T^*)$, and $\lim_{t\rightarrow T^*}M(t)=\infty$. Then for all
$t_0\in(0,T^*),$ we can define
$$t_0^+:=t^+(t_0):=\max\{t\in(t_0,T^*):\;M(t)=2M(t_0)\}.$$
Choose $A\geq1$ and let
\begin{equation}\label{Lamb+}
  \lambda(t_0):=\left(\frac{1}{2A}M(t_0)\right)^{-1/(2m\alpha_1)}.
\end{equation}
we claim that
\begin{equation}\label{Res+}
    \lambda^{-2m}(t_0)(t_0^+-t_0)\leq D,\qquad
    t_0\in\left(\frac{T^*}{2},T^*\right),
\end{equation}
where $D>0$ is a positive constant which does not depend on $t_0.$\\
We proceed by contradiction. If (\ref{Res+}) were false, then there
would exist a sequence $t_n\rightarrow T^*$ such that
$$\lambda_n^{-2m}(t_n^+-t_n)\longrightarrow\infty,$$
where $\lambda_n=\lambda(t_n)$ and $t_n^+=t^+(t_n).$ For each $t_n$
choose
\begin{equation}\label{Seq+}
(\hat{x}_n,\hat{t}_n)\in\mathbb{R}^n\times(0,t_n]\quad\hbox{such
that}\quad |u(\hat{x}_n,\hat{t}_n)|\geq\frac{1}{2}M(t_n).
\end{equation}
Obviously, $M(t_n)\rightarrow\infty;$ hence,
$\hat{t}_n\rightarrow T^*.$ Next, rescale the function $u$ as
\begin{equation}\label{fu+}
\varphi^{\lambda_n}(y,s):=\lambda_n^{2m\alpha_1}u(\lambda_n
y+\hat{x}_n, \lambda_n^{2m} s+\hat{t}_n),\qquad
(y,s)\in\mathbb{R}^n\times I_n(T^*),
\end{equation}
where $I_n(t):=(-\lambda_n^{-{2m}}\hat{t}_n,
\lambda_n^{-{2m}}(t-\hat{t}_n))$ for all $t>0.$ Then
$\varphi^{\lambda_n}$ is a mild solution of
\begin{equation}\label{Nsol+}
\varphi_s+(-\Delta)^m\varphi=\int_{-\lambda^{-2m}_n\hat{t}_n}^s(s-r)^{-\gamma}|\varphi(r)|^p\,dr\qquad
\hbox{in}\;\mathbb{R}^n\times I_n(T^*).
\end{equation}
On the other hand, $|\varphi^{\lambda_n}(0,0)|\geq A$, and
$$|\varphi^{\lambda_n}|\leq\lambda_n^{2m\alpha_1}M(t^+_n)=
\lambda_n^{2m\alpha_1}2M(t_n)=4A\qquad\hbox{in}\;\;\mathbb{R}^n\times
I_n(t_n^+),$$ thanks to \eqref{Lamb+} and the definition of $t_n^+$.

Moreover, as
$$\varphi^{\lambda_n}\in C([-\lambda_n^{-2}\hat{t}_n,T],C_0(\mathbb{R}^N)\cap L^1(\mathbb{R}^N))\quad\text{for all}\;T\in I_n(T^*),$$
so, as in Lemma \cite[Lemma~4.2]{FinoK}, $\varphi^{\lambda_n}$ is a weak solution of (\ref{Nsol+}).\\
By the maximal regularity theory \cite[Theorem~2]{Dongkim}, we have
\[
\varphi^{\lambda_n}\in W^{1,q}((-\lambda_n^{-2m}\hat{t}_n,T);L^q(\mathbb{R}^N))\cap L^q((-\lambda_n^{-2m}\hat{t}_n,T);W^{2m,q}(\mathbb{R}^N)),
\]
for any $q\in(1,\infty)$. Therefore, from the uniform interior Schauder's estimates (see \cite{Boccia}), the
$C^{2m+\mu,1+\mu/2m}_{loc}(\mathbb{R}^n\times\mathbb{R})$-norm
of $\varphi^{\lambda_n}$ is uniformly bounded, for some $\mu\in(0,1)$. Hence, we obtain a
subsequence converging in
$C^{2m+\mu,1+\mu/2m}_{loc}(\mathbb{R}^n\times \mathbb{R})$ to a solution $\varphi$ of
\[
\varphi_s+(-\Delta)^m\varphi=C_\alpha\,I^\alpha_{-\infty|s}(|\varphi|^p)\qquad\text{in}\;
\mathbb{R}^n\times(-\infty,+\infty),
\]
such that $|\varphi(0,0)|\geq A$ and $|\varphi|\leq 4A$ in $\mathbb{R}^n\times\mathbb{R}.$ Whereupon, 
using Lemma \ref{lemma1+}, we infer that $\varphi\equiv0$ in $\mathbb{R}^n\times(-\infty,+\infty).$ Contradiction
with the fact that $|\varphi(0,0)|\geq A\geq1.$ This proves (\ref{Res+}).\\
 Next we use an idea from Hu \cite{Hu}. From (\ref{Lamb+}) and
 (\ref{Res+}) it follows that
\[
(t_0^+-t_0)\leq
 D(2A)^{1/\alpha_1}M(t_0)^{-1/\alpha_1}\qquad\hbox{for
 any}\;t_0\in\left(\frac{T^*}{2},T^*\right).
 \]
 Fix $t_0\in\left(T^*/2,T^*\right)$ and denote
 $t_1=t_0^+,t_2=t_1^+,t_3=t_2^+,\dots$. Then
 \begin{eqnarray*}
   t_{j+1}-t_j&\leq&D(2A)^{1/\alpha_1}M(t_j)^{-1/\alpha_1}, \\
   M(t_{j+1})&=&2M(t_j),
 \end{eqnarray*}
 $j=0,1,2,\dots .$ Consequently,
 \begin{eqnarray*}
   T^*-t_0&=&\sum^{\infty}_{j=0}(t_{j+1}-t_j)\leq D(2A)^{1/\alpha_1}
   \sum^{\infty}_{j=0}M(t_j)^{-1/\alpha_1}\\
    &=&D(2A)^{1/\alpha_1}M(t_0)^{-1/\alpha_1}\sum^{\infty}_{j=0}2^{-j/\alpha_1}.
 \end{eqnarray*}
Finally, we conclude that
\[
|u(x,t_0)|\leq M(t_0)\leq C(T^*-t_0)^{-\alpha_1},\qquad\forall\;t_0\in(0,T^*)
\]
where
\[
C=2A\left(D\sum^{\infty}_{j=0}2^{-j/\alpha_1}\right)^{\alpha_1};
\]
so
\[
\sup_{\mathbb{R}^n}|u(\cdotp,t)|\leq C(T^*-t)^{-\alpha_1},\qquad\forall\;t\in(0,T^*).
\]

\noindent $\bullet$ \underline{The lower blow-up rate estimate}. If we repeat
the proof of the local existence of Theorem \ref{T0+}, by
taking $\|u\|_1\leq \theta$ instead of $\|u\|_1\leq
2\|u_0\|_\infty$ in the space $E_T$ for all positive constant
$\theta>0$ and all $0<t< T,$ then the condition (\ref{conditionssurT+}) of
$T$ will be:
\begin{equation}\label{esti6+}
\|u_0\|_\infty+C T^{2-\gamma}\theta^p\leq\theta,
\end{equation}
and then, like before, we infer that $\|u(t)\|_\infty\leq
\theta$ for (almost) all $0<t<T.$ Consequently, if $\|u_0\|_\infty+C
t^{2-\gamma}\theta^p\leq\theta,$ then $\|u(t)\|_\infty\leq\theta.$
Applying this to any point in the trajectory, we see that if $0\leq
s<t$ and
\begin{equation}\label{esti7+}
(t-s)^{2-\gamma}\leq\frac{\theta-\|u(s)\|_\infty}{C\theta^p},
\end{equation}
then $\|u(t)\|_\infty\leq\theta,$ for all $0<t<T.$\\
Moreover, if $0\leq s<T^*$ and $\|u(s)\|_\infty<\theta,$ then:
\begin{equation}\label{esti8+}
(T^*-s)^{2-\gamma}>\frac{\theta-\|u(s)\|_\infty}{C\theta^p}.
\end{equation}
Indeed, arguing by contradiction and assuming that for some
$\theta>\|u(s)\|_\infty$ and all $t\in(s,T^*)$ we have
\[
(t-s)^{2-\gamma}\leq\frac{\theta-\|u(s)\|_\infty}{C\theta^p}.
\]
Then, using (\ref{esti7+}), we infer that
$\|u(t)\|_\infty\leq\theta$ for all $t\in(s,T^*);$ this contradicts the fact that
$\|u(t)\|_\infty\rightarrow\infty$ as
$t\rightarrow T^*.$\\
Next, for example, by setting $\theta=2\|u(s)\|_\infty$ in
(\ref{esti8+}), we see that for $0<s<T^*$ we have:
\[
(T^*-s)^{2-\gamma}>C'\|u(s)\|_\infty^{1-p},
\]
and by the continuity of $u$ we get
\begin{equation}\label{esti9+}
c(T^*-s)^{-\alpha_1}<\sup_{x\in\mathbb{R}^n}|u(x,s)|,\qquad\forall\;s\in(0,T^*).
\end{equation}
\hfill$\square$

\subsection*{Acknowledgment}
The author is supported by the Lebanese University research program.

\bibliographystyle{amsplain}

\begin{thebibliography}{10}

\bibitem{tedeev} D. Andreucci, A. F. Tedeev, \textit{Universal bounds at the blow-up time for nonlinear parabolic equations},
Adv. Differential Equations \textbf{10} (2005), no. 1, 89--120.

\bibitem{Baras} P. Baras, R. Kersner, \textit{Local and global solvability of a class of semilinear parabolic equations.}
J. Differential Equations \textbf{68} (1987), no. 2, 238--252.

\bibitem{BarasPierre}P. Baras, M. Pierre, \textit{Crit\`ere d'existence de solutions positives
pour des \'equations semi-lin\'eaires non monotones}, Ann. Inst. H. Poincar\'e Anal. Non Lin\'eaire \textbf{2} (1985), 185--212.

\bibitem{Boccia}S. Boccia, \textit{Schauder estimates for solutions of high-order parabolic systems}, Methods
Appl. Anal. \textbf{20} (2013), no. 1, 47--67.


\bibitem {CH}  T. Cazenave, A. Haraux, \textit{Introduction aux probl\`emes
d'\'evolution semi-lin\'eaires}, Ellipses, Paris, (1990).

\bibitem{CDW} T. Cazenave, F. Dickstein, F. D. Weissler, \textit{An equation
whose Fujita critical exponent is not given by scaling}, Nonlinear
Analysis \textbf{68} (2008), 862--874.

\bibitem {CFila}  M. Chlebik, M. Fila, \textit{From critical exponents to
blow-up rates for parabolic problems},
Rend. Mat. Appl. (7) 19 \textbf{4} (1999), 449--470.

\bibitem {Cui}  S. Cui, \textit{Local and global existence of solutions to semilinear parabolic initial value problems}, Nonlinear Analysis 43 (2001), 293--323.


\bibitem{Dongkim} H. Dong, D. Kim, \textit{On the $L^p$-Solvability of higher order parabolic and elliptic systems with BMO coefficients}, Arch. Rational Mech. Anal. 199 (2011), 889--941.

\bibitem{FQ} M. Fila, P. Quittner, \textit{The Blow-Up Rate for a Semilinear Parabolic System}, J.
of Mathematical Analysis and Applications \textbf{238} (1999), 468--476.


 \bibitem{FinoK} {A. Z. Fino, M. Kirane}, \textit{Qualitative properties of solutions to a time-space fractional evolution equation}, J. Quarterly of Applied Mathematics \textbf{70}(2012), 133--157.
 
  \bibitem{Furati} {K. M. Furati, M. Kirane}, \textit{Necessary Conditions for the Existence of Global Solutions to Systems of Fractional Differential Equations}, Fractional Calculus and Applied Analysis 11(3) (2008), 281--298.
 
  \bibitem{Folland} {Gerald B. Folland}, \textit{Real analysis: modern techniques and their applications}, second ed., John Wiley \& Sons, 1999.
  
\bibitem{Fuj} H. Fujita, \textit{On the blowing up of solutions of the
problem for $u_t=\Delta u+u^{1+\alpha},$} J. Fac. Sci. Univ. Tokyo
\textbf{13} (1966), 109--124.

\bibitem{Galaktionov} V.A. Galaktionov, S.I. Pohozaev, \textit{Existence and blow-up for higher-order semilinear parabolic equations: Majorizing order-preserving operators},
Indiana Univ. Math. J. 51(6) (2002), 1321--1338.

\bibitem{Giga} Y. Giga, R. V. Kohn, \textit{Characterizing blow-up using similarity variables}, Indiana Univ.
Math. J. 36 (1987), 1--40.

\bibitem{11}  K. Hayakawa, \textit{ On nonexistence of global solutions of
some semilinear parabolic differential equations}, Proc. Japan Acad. 4 (1973), 503--505.

\bibitem{Hu} B. Hu, \textit{Remarks on the blow-up estimate for solutions of the heat equation
with a nonlinear boundary condition}, Differential Integral Equations
9 (1996), 891--901.

\bibitem{KQ} M. Kirane, M. Qafsaoui, \textit{Global nonexistence for the Cauchy problem of some nonlinear
Reaction-Diffusion systems}, J. Math. Analysis and Appl. \textbf{268} (2002), 217--243.


\bibitem{PM1} E. Mitidieri, S. I. Pohozaev,  \textit{A priori estimates and
blow-up of solutions to nonlinear partial differential equations and
inequalities}, Proc. Steklov. Inst. Math. \textbf{234} (2001), 1--383.

\bibitem{PX} Hongjing Pan, Ruixiang Xing, \textit{Blow-up rates for higher-order semilinear parabolic equations and
systems and some Fujita-type theorems}, J. Math Anal. Appl. \textbf{339} (2008), 248--258.
 
 \bibitem{Pazy} A. Pazy, \textit{Semigroups of Linear Operators and Applications to Partial Differential Equations},
Springer-Verlag, 1983.
 
\bibitem{Peletier} L. A. Peletier, W. C. Troy, \textit{Spatial Patterns: Higher Order Models in
Physics and Mechanics}, Progress in Nonlinear Differential Equations and their Applications, 45,
Birkh\"{a}user Boston Inc., Boston,MA, 2001.
 
\bibitem{SKM} S. G. Samko, A. A. Kilbas, O. I. Marichev, \textit{Fractional
integrals and derivatives}, Theory and Applications, Gordon and
Breach Science Publishers, 1987.


 \bibitem{SS} {F. Sun, P. Shi}, \textit{Global existence and non-existence for a higher-order parabolic
equation with time-fractional term}, J. Nonlinear analysis \textbf{75}(2012), 4145--4155.


\bibitem{Yuta} Y. Wakasugi,  \textit{On the diffusive structure for the damped wave equation with variable coefficients}, Doctoral thesis, Osaka University, 2014.

\bibitem{Zhang} Qi S. Zhang, \textit{A blow up result for a nonlinear wave equation with damping:
the critical case}, C. R. Acad. Sci. Paris, Vol. ${\bf 333}$
$(2001),$ no. $2,$ 109--114.
\end{thebibliography}

\end{document}